\documentclass{article}
\usepackage[utf8]{inputenc}
\usepackage{authblk}
\usepackage{setspace}
\usepackage[margin=1.25in]{geometry}
\usepackage{graphicx}
\graphicspath{ {./figures/} }
\usepackage{subcaption}
\usepackage{amsmath}
\usepackage{amssymb}
\usepackage{amsthm}
\usepackage{thmtools}
\usepackage{tikz}
\usepackage{tikz-cd}
\usetikzlibrary{calc}
\usepackage[utf8]{inputenc}
\usepackage{bbold}
\usepackage[shortlabels]{enumitem}
\usepackage[hidelinks]{hyperref}      
\usepackage[nameinlink]{cleveref}

\usetikzlibrary{patterns}
\tikzset{
pattern size/.store in=\mcSize, 
pattern size = 5pt,
pattern thickness/.store in=\mcThickness, 
pattern thickness = 0.3pt,
pattern radius/.store in=\mcRadius, 
pattern radius = 1pt}
\makeatletter
\pgfutil@ifundefined{pgf@pattern@name@redstripes}{
\pgfdeclarepatternformonly[\mcThickness,\mcSize]{redstripes}
{\pgfqpoint{0pt}{-\mcThickness}}
{\pgfpoint{\mcSize}{\mcSize}}
{\pgfpoint{\mcSize}{\mcSize}}
{
\pgfsetcolor{\tikz@pattern@color}
\pgfsetlinewidth{\mcThickness}
\pgfpathmoveto{\pgfqpoint{0pt}{\mcSize}}
\pgfpathlineto{\pgfpoint{\mcSize+\mcThickness}{-\mcThickness}}
\pgfusepath{stroke}
}}
\makeatother

\newcommand{\ceil}[1]{\left\lceil #1 \right\rceil}
\newcommand{\prs}[1]{\left( #1 \right)}
\newcommand{\crb}[1]{\left\{ #1 \right\}}

\tikzset{
pattern size/.store in=\mcSize, 
pattern size = 5pt,
pattern thickness/.store in=\mcThickness, 
pattern thickness = 0.3pt,
pattern radius/.store in=\mcRadius, 
pattern radius = 1pt}
\makeatletter
\pgfutil@ifundefined{pgf@pattern@name@_7la7libyc}{
\pgfdeclarepatternformonly[\mcThickness,\mcSize]{_7la7libyc}
{\pgfqpoint{0pt}{0pt}}
{\pgfpoint{\mcSize+\mcThickness}{\mcSize+\mcThickness}}
{\pgfpoint{\mcSize}{\mcSize}}
{
\pgfsetcolor{\tikz@pattern@color}
\pgfsetlinewidth{\mcThickness}
\pgfpathmoveto{\pgfqpoint{0pt}{0pt}}
\pgfpathlineto{\pgfpoint{\mcSize+\mcThickness}{\mcSize+\mcThickness}}
\pgfusepath{stroke}
}}
\makeatother
\tikzset{every picture/.style={line width=0.75pt}}

\theoremstyle{plain}
\newtheorem{theorem}{Theorem}[section]
\newtheorem{lemma}[theorem]{Lemma}
\newtheorem{proposition}[theorem]{Proposition}
\newtheorem{corollary}[theorem]{Corollary}

\theoremstyle{definition}
\newtheorem{definition}[theorem]{Definition}
\newtheorem{example}[theorem]{Example}
\newtheorem{construction}[theorem]{Construction}

\theoremstyle{remark}
\newtheorem{remark}[theorem]{Remark}

\usepackage[backend=bibtex, style=alphabetic, citestyle=alphabetic]{biblatex}
\newcommand{\C}{\mathbb{C}}
\newcommand{\R}{\mathbb{R}}
\newcommand{\Z}{\mathbb{Z}}
\newcommand{\N}{\mathbb{N}}

\newcommand{\Q}{\mathbb{Q}}
\newcommand{\di}{dimension}

\title{Multi-Dimensional Cohomological Phenomena in the Lower Multiparametric Model}

\author{Jon V. Kogan\footnote{Department of Mathematics, Hebrew University, Jerusalem 91904, Israel. e-mail:
jonatan.kogan$@$mail.huji.ac.il. Supported by ERC grant 3012006831 and ISF grant 3011003773}}

\date{\vspace{-5ex}}

\begin{document}

\maketitle

%%%%%% Abstract %%%%%%

\begin{abstract}
In the past two decades, extensive research has been conducted on the (co)homology of various models of random simplicial complexes. 
So far, it
has always been examined merely as a list of groups. This paper expands upon this by describing both the ring structure and the Steenrod-algebra structure of the cohomology of the lower multiparametric model.
We prove that the ring structure is always a.a.s trivial, while, for certain parameters, the Steenrod-algebra a.a.s acts non-trivially.
This reveals that complex multi-dimensional topological structures appear as subcomplexes of this model.
\end{abstract}
\section{Introduction}
\subsection{Background}
Random simplicial complexes and their topological properties have been an active area of study in recent years (see, for instance, \cite[chapter 22]{HandbookDnCG3}, or \cite{Bobrowski_2022}). A \textbf{model} of random complexes is an assignment of a probability distribution on simplicial complexes with $n$ vertices, for each $n$.
Two such models that have received extensive study are the clique complex of a random graph from $G(n,p)$ (\cite{KAHLE09}), and the Linial-Meshulam-Wallach model, first appearing in \cite{MeshulamWallach} 
(the 2-\di al version first appearing in \cite{LinialMeshulam}), comprised of the complete $m$-skeleton on $n$ vertices, and a random set of $m+1$-simplices.
Interpolating between these is the \textbf{Lower Multiparametric Model}:
\begin{definition}\label{defLMM}
Define the distribution $X(n;p_1,p_2,...)$ on hypergraphs with $n$ vertices by including faces independently, where facets of \di \ $i$ are included with probability $p_i$.
This results in a distribution on simplicial complexes, $\underline{X}(n;p_1,p_2,...)$, defined by sending a hypergraph from $X(n;p_1,p_2,...)$ to the maximal simplicial complex contained in it. This distribution is called the \textbf{Lower Multiparametric Model} or the \textbf{Lower Model} for short.
 \end{definition}
$\underline{X}(n;p_1,p_2,...)$ was first introduced in \cite{CostaFarber}, and may be thought of as the distribution on simplicial complexes with $n$ vertices, where each edge is included independently with probability $p_1$, each triangle whose boundary was formed in the first step is filled independently with probability $p_2$, each tetrahedron boundary gets filled with probability $p_3$ and so on. \par
 
 $\underline{X}(n;p_1,p_2,...)$ interpolates between the clique complex and the Linial-Meshulam-Wallach model (with parameters $(p,1,1,...),(1,1,...,1,p_r,0,...)$ respectively). 
 In both these models, cohomology is concentrated in one \di \ (see \cite{Bobrowski_2022}). However, in \cite{Fowler19}, Fowler studied the asymptotic behavior of the cohomology of the lower multiparametric model with coefficients in $\Q$, and found that multiple cohomologies can be non-zero for the same parameters. By "a.a.s" we mean "asymptotically almost surely"- with probability approaching 1 as $n\to \infty$. See \Cref{aas} for the precise definition.
 \begin{theorem}[Fowler]\label{thmFowler}
    Let $p_i=n^{-\alpha_i}$, $\alpha_i\in\R_{\ge 0}$. Denote: $$S_1^k(\{\alpha_i\}_i)=\sum_{i=1}^{k+1}\binom{k+1}{i}\alpha_i\qquad S_2^k(\{\alpha_i\}_i)= \sum_{i=1}^{k}\binom{k+2}{i+1}\alpha_i$$
        Then:
        \begin{itemize}
            \item If $S_1^k(\{\alpha_i\}_i)<1$, then $H^k(\underline{X},\Q)=0$ a.a.s.
            \item If $S_2^k(\{\alpha_i\}_i)>k+2$, then $H^k(\underline{X},\Z)=0$ a.a.s.
            \item If $S_1^k(\{\alpha_i\}_i)\geq 1$ and $S_2^k(\{\alpha_i\}_i)<k+2$ and all $\alpha_i$ are positive, then $H^k(\underline{X},\Q)\neq 0$ a.a.s.
        \end{itemize}
\end{theorem} 

This result exhibits a phenomenon in the existing literature- most if not all papers in the field discuss cohomology merely as a list of Betti numbers. Cohomology (with ring coefficients), on the other hand, is a graded ring, with the multiplication operation being the cup product.
Moreover, cohomology with coefficients in $\Z/p$ also has a structure of a Steenrod algebra (\cite[chapter 4.L]{AT}) which is the algebra of stable cohomology operations 
%(that is natural transformations between cohomology functors that commute with the equivalence $H^{*}(\Sigma X)=H^{*-1}(X)$) 
over $\Z/p$, generated by the Steenrod squares $Sq^i:H^*\to H^{*+i}$ in the case of $p=2$, or the Steenrod powers $P^i:H^*\to H^{*+2i(p-1)}$ and the Bockstein homomorphism $\beta:H^*\to H^{*+1}$ in the case of an odd $p$. 
These additional structures are useful in distinguishing between spaces with otherwise identical cohomology.
The following are classic examples:
\begin{example}
    Let $X=S^2\vee S^1\vee S^1$ and $ Y=S^1\times S^1$, the 2-\di al torus. Both have identical cohomology groups $H^0=\Z, H^1=\Z^2, H^2=\Z$. However, $H^*(X)$ has a trivial multiplicative structure, whereas in $H^*(Y)$ the  product of the two generators of $H^1$ equals the generator of $H^2$, proving that $X$ and $Y$ are not homotopy equivalent.
\end{example}
\begin{example}
    Let us compare $\C P^2$ and $Y=S^2\vee S^4$. Again, these spaces have identical cohomology over any ring (due to their cell structure), yet the former has a non-trivial cup product while the latter does not.\par
    Furthermore, in $\C P^2$ the generator of $H^2$ squares to the generator of $H^4$. In $\Z/2$ coefficients, this coincides with $Sq^2$, which, unlike the cup product, is stable. 
    Thus, the spaces $\Sigma\C P^2$ and $\Sigma Y=S^3\vee S^5$, which can no longer be distinguished by the cup product structure (as it is trivial for both), may still be distinguished by the non-triviality of $Sq^2$ on the former.
\end{example}
We thus see that this additional algebraic structure is useful, in particular to distinguish a space from a wedge of spheres.
In \cite{FarberMeadNowik22}, the lower model is described as being "homologically approximated by a wedge of spheres", and being a wedge of spheres is a common conjecture/result in the field of simplicial complexes in general (see for instance \cite[\nopp 13.15,17.28,19.10]{Kozlov08}).\par
The goal of this paper is to determine whether these structures are trivial or not for the lower multiparametric model. To the author's knowledge, this is the first time that non-trivial multidimensional phenomena are studied in stochastic topology.

\subsection{Results}
\begin{definition}
    For a topological space $X$ and a ring $R$, the \textbf{cup length} of $X$ is the maximal number of elements
    
    $\{a_i\}_{i}$ where $a_i \in H^{n_i}(X; R)$ and $n_i > 0$ for all $i$, such that $\prod_{i} a_i \neq 0$.
\end{definition}
\begin{theorem}\label{thmcup}  
 Given a list of positive real numbers $\{\alpha_i \}_{i=1} ^\infty$ that satisfy  $S_1 ^k(\{\alpha_i\}_i) \neq 1$ for all $k\in \N$, the family $\underline{X}(n;\{n^{-\alpha_i}\}_i)$ a.a.s has a cup length of 1 over $\Q$. Moreover, the image of the cup product in $H^{2l}$ is trivial over any ring of coefficients, where $l\in \N$ is any number such that $S^l _1(\{\alpha_i\}_i)>1$.
\end{theorem}
Thus, for a full measure of the $\alpha$'s, no non-trivial cup product exists. This is proven by noticing that the requirement that both $H^{k},H^{2k}$ are non-zero simultaneously is very restrictive in this model.
These restrictions limit the possible homotopy types of strongly connected components (see \Cref{defStrongC}), in which the cup and Steenrod operations are concentrated (in a sense made concrete in \Cref{lemCupSteenConc}).\par
Unlike the cup product, Steenrod operations are stable and thus each one requires only a fixed difference of \di s of cohomologies to exist (as opposed to a ratio of 2), so the restriction is less severe. 
\begin{theorem}\label{thmSteenrod}
  Fix a non-zero element $\sigma$ of the Steenrod algebra over $\Z/p$, where $p$ is a prime. Then there exist $D\in \N$ and a set $A\subset \R_{\geq 0} ^D$ with non empty interior such that for $(\alpha_1,...,\alpha_D)\in A$ and $p_i=0,i>D$, $\sigma$ is a.a.s non-zero as a map on $H^*(\underline{X}(n;\{n^{-\alpha_i}\}),\Z/p)$.
  \end{theorem} 
  To prove this, we construct a variation of a simplicial model of the suspension operation (see \Cref{conKoganSus}). This construction modifies a complex into one that is more likely to appear as a subcomplex of $\underline{X}$. 
  If we apply this construction sufficiently many times, the resulting complex appears as a subcomplex of $\underline{X}(n;\{n^{-\alpha_i}\})$ a.a.s. 
  This construction preserves several desirable properties of the complex, including having a non-zero Steenrod operation and strong connectivity. Thus, for some range of $\alpha$’s, $\sigma$ is non-zero on the complex.\par
This indicates that the lower model has a highly non-trivial topological structure for certain parameters.
\subsection{Structure of the paper}

In \Cref{prelim}, we gather definitions and results useful in the analysis of $\underline{X}$ and the topology of simplicial complexes in general. 
The proof of \Cref{thmcup} is separated into \Cref{cup1}, which addresses the case where $H^1(\underline{X})\ne 0$, and \Cref{cupK} for other cases. The separation is due to the proofs being distinct in structure.
\Cref{steenChap} is devoted to the proof of \Cref{thmSteenrod}.
Lastly, \Cref{discus} details problems left open by the present manuscript.

\section{Preliminaries}\label{prelim}
\subsection{Topology \& Probability}
\begin{definition}
    A simplicial complex is called \textbf{pure $d$-dimensional}  if any face in the complex is contained in a $d$-\di al simplex. Note that this implies that $d$ is the maximal dimension.
\end{definition}

\begin{definition}\label{defStrongC}
 For $d$-dimensional simplexes $\sigma,\tau$ in a simplicial complex $X$, define a relation $\sim$ on $d$-faces of $X$, where  $\sigma \sim \tau$ if they share a $(d-1)$-dimensional face. 
 This is obviously reflexive and symmetric. Complete this relation under transitivity and call it $\sim'$. A \textbf{Strong Connectivity Component} is an equivalence class of $d$-dimensional faces of $X$ under this relation. 
A pure-$d$-dimensional simplicial complex is called \textbf{Strongly Connected} if all of its $d$-\di al faces are in the same component.
\end{definition}

\begin{definition}
 Let $a,b$ be two cochains of degrees $k,l$ respectively. Then the cup product $a\cup b$ is a cochain of \di \ $k+l$ defined on a simplex $\Delta^{k+l}=(v_0,...,v_{k+l})$ by $$(a\cup b)(v_0,...,v_{k+l})=a(v_0,...,v_{k})\cdot b(v_{k},...,v_{k+l})$$
 This produces a graded product structure on cohomology (see \cite{AT}). 
The important thing for us is that the value of the cup product on a particular simplex depends only on the values of the multiplied cochains on the faces of that simplex.\end{definition}
\begin{remark}
    In this paper we call a cup product "non-trivial" on a particular space only when there exist two cocycles of positive dimension whose product is non-zero in cohomology, or equivalently when the cup-length is strictly greater than 1. 
    For a non-empty, connected space $X$, $H^0(X;R)=R$, and multiplication by a 0-\di al cocycle is equivalent to multiplication by a scalar. Thus, omitting the condition of the cocycles being of positive dimension would make the question uninteresting. The disconnected case is similar.
\end{remark}
We now, for completeness' sake, give a non-constructive definition of the Steenrod Algebra (for an explicit construction, see for instance \cite[chapter 4.L]{AT}). The reason for this is that almost no detail of any definition/construction will be used.
\begin{definition}
Let $p$ be a prime. A \textit{cohomology operation} over $\mathbb{Z}/p$ is a natural transformation between functors $\Theta: H^n(-; \mathbb{Z}/p) \to H^m(-; \mathbb{Z}/p)$.

A sequence of cohomology operations 
\[ \Theta_n: \tilde{H}^n(-; \mathbb{Z}/p) \to \tilde{H}^{n+k}(-; \mathbb{Z}/p) \]
defined for all integers $n \geq 0$ (where $k \geq 0$ is a fixed integer called the degree of the operation) is said to be \textbf{stable} if it commutes with the suspension isomorphism. Explicitly, if 
 $\sigma: \tilde{H}^n(X; \mathbb{Z}/p) \xrightarrow{\cong} \tilde{H}^{n+1}(\Sigma X; \mathbb{Z}/p)$ is the natural suspension isomorphism, the following diagram commutes for all $n$:
\[
\begin{tikzcd}[column sep=large, row sep=large]
\tilde{H}^n(X; \mathbb{Z}/p) \arrow[r, "\sigma", "\cong"'] \arrow[d, "\Theta_n"'] & \tilde{H}^{n+1}(\Sigma X; \mathbb{Z}/p) \arrow[d, "\Theta_{n+1}"] \\
\tilde{H}^{n+k}(X; \mathbb{Z}/p) \arrow[r, "\sigma", "\cong"'] & \tilde{H}^{n+k+1}(\Sigma X; \mathbb{Z}/p)
\end{tikzcd}
\]
The \textbf{mod-$p$ Steenrod Algebra}, $A_p$, is the algebra of stable cohomology operations under addition and composition.

For $\sigma\in A_p$, we denote by $deg(\sigma)$-the degree of $\sigma$- the amount by which it increases \di . 
\end{definition}
In our proofs, we will use two properties of the mod-$p$ Steenrod Algebra:
\begin{enumerate}
    \item The fact that the operations are stable.
    \item The fact that it is generated, in the $p=2$ case by $\crb{Sq^i}_{i=0}^\infty$, and in the odd $p$ case by $\beta$ and $\crb{P^i}_{i=0}^\infty$.
    For a cocycle $c$ and $\sigma\in \crb{\beta,Sq^i,P^i}$, the value of $\sigma(c)$ on a particular simplex $s$ is a function of the values of $c$ on the faces of $s$.

    The explicit formulae for these functions can be seen in \cite[chapter 4.L]{AT} for the $p=2$ case (which are similar to the cup product formula above), and in \cite{Kaufmann_2021} for other $p$.
\end{enumerate}

We now turn to the key insight that enables the results in this paper:
\begin{lemma}\label{lemCupSteenConc}
Let $X$ be a simplicial complex, $d\in\N$, $p$ a prime and $\sigma\in A_p$.
    \begin{enumerate}
    \item Any one of $H^d(X)$,  $Im (\cup : H^*(X)\times H^*(X)\to H^d(X))$  and $Im(\sigma)\subset H^d$ is non-trivial only if there exists a $d$-\di al strong-connectivity component on which it is non-zero. 
    If $X$ has $d$ as the maximal \di , this is an if and only if for $H^d(X)$.
    \item Let $0\neq p,q,r,s\in H^*(X)$ be elements satisfying either $p\cup q=r$ or $\sigma (p)=s$. Let $X\xrightarrow{i}C=X\cup_{\partial\Delta} \Delta$- a complex with an additional simplex attached along the entire boundary.
    Then $(i^*)^{-1}(r)$ (or $s$) is non-empty if and only if $[r]$ (or $[s]$) can be expanded as a cochain to a cocycle in $C$. In addition, if there exist extensions of $p,q$ to cocycles $p',q'$ in $C$, $i^*(p'\cup q')=r$ (or $i^* (\sigma p')=s$).
    \end{enumerate}
\end{lemma}
\begin{proof}
    \begin{enumerate}
    \item Denote the $d$-\di al strong connectivity components of $X$ by $\{X_i\}_i$. Define $f:\bigsqcup_i X_i\to X$ to be the map which restricts to the inclusions $f_i:X_i\to X$. 

    Let $0\ne [h]\in H^d(X)$ be a cocycle. For the sake of contradiction, assume $\forall i\ f_i^*(h):=h_i=0$, i.e that it is a coboundary. 
    Therefore, there exists a cochain $c_i$ s.t $\delta(c_i)=h_i$. 

    $f$ is injective on the $(d-1)$-skeletons, and bijective on $d$-faces, so $\bigsqcup_i c_i$ may be extended to a cochain $c$ on $X$ (the $(d-1)$-faces outside of the image are not part of any $d$-face), where $\delta(c)=h$. This contradicts our assumption that $[h]\ne 0$, proving there exists an $i$ s.t $[f^*_i(h)]\ne 0$.

    If $h=a\cup b$ (or $h=\sigma(g)$ for some $\sigma\in A_p$), we have, for every $i$, that $f_i^*(h)=f_i^*(a)\cup f_i^*(b)$ (or $f_i^*(h)=\sigma f_i^*(g)$) by naturality of the cup product and Steenrod algebra, granting the results for these as well.

    Now, if $d$ is the maximal \di\ of $X$, any non-trivial $d$-cocycle on $X_i$ can be extended to a non-trivial cocycle on $X$, finishing the proof.

    \item The claim about $r,s$ is obvious, and the second claim follows from naturality of the cup product and Steenrod operations.
    \end{enumerate}
\end{proof}
\begin{corollary}
    For any simplicial complex $X$, the non-triviality of any of $H^d(X)$, $\cup:H^*(X)\times H^*(X)\to H^d(X)$ or $\sigma:H^{d-\deg(\sigma)}(X)\to H^d(X)$ implies its non-triviality for $sk^d (X)$. 
    Furthermore, to deduce the non-triviality of $H^d$, $\cup$ or $\sigma$ on the entire complex from that on a strong connectivity component, we need only check if the relevant cocycles can be extended. In particular, only the addition of simplexes of \di s $\dim(p)+1, \dim(q)+1$ or $\dim(r)+1$ (or $\dim(s)+1$) might affect these extensions.
\end{corollary}

\begin{definition}\label{aas}    
 Let $D(n;\vec{p}(n))$ be a family of distributions depending on a natural number $n$ and perhaps other parameters. $D(n;\vec{p}(n))$ satisfies a property $q$ \textbf{a.a.s} (asymptotically almost surely) if $\underset{n\to \infty}{\lim} P(q)=1$.
 \end{definition}

\begin{example}
    
Let $G(n,p)$ be a distribution on graphs with $n$ vertices, where each edge appears independently with probability $p$.
A theorem by Erdős and Rényi states that $\frac{\ln(n)}{n}$ is the threshold function for connectivity, meaning that for $p(n)=o(\frac{\ln(n)}{n})$ the graph is a.a.s disconnected, and for $p(n)=\omega(\frac{\ln(n)}{n})$ it is a.a.s connected. \end{example} 

\subsection{General Results in the Lower Multiparametric Model}

\begin{proposition}
    Let $X\in \underline{X}(n;n^{-\alpha_1},...)$ be a random simplicial complex, where $\crb{\alpha_i}$ are non-negative real numbers. Then a.a.s, for any $0\neq a,b\in H^*(X;\Q)$ cocycles of pure dimension, $\frac{1}{2}\leq \frac{\dim(a)}{\dim(b)}\leq 2$.  
    
\end{proposition} 
\begin{proof}
    From \Cref{thmFowler} we know that any $k$ for which rational cohomology is not a.a.s zero satisfies
    \begin{equation}\label{eq:S1S2ineq}
            \sum_{i=1} ^{k+1}\binom{k+1}{i}\alpha_i\geq 1\ \land\ \sum_{i=1} ^{k}\binom{k+2}{i+1}\alpha_i< k+2.
    \end{equation}
    Let $k_0$ be the minimal $k$ where the left-hand condition is satisfied. By substituting $k=2k_0+1$ in the left-hand expression we get $$\sum_{i=1} ^{2k_0+1}\binom{2k_0+3}{i+1}\alpha_i<2k_0+3\Longleftrightarrow \sum_{i=1} ^{2k_0+1}\frac{\prod_{j=1}^{i}(2k_0+3-j)}{(i+1)!}\alpha_i<1.$$
    This contradicts the left-hand inequality in \Cref{eq:S1S2ineq}, since it would mean in particular 
    $$1\leq \sum_{i=1} ^{k_0+1}\binom{k_0+1}{i}\alpha_i\leq\sum_{i=1} ^{k_0+1}\frac{\prod_{j=1}^{i}(2k_0+3-j)}{(i+1)!}\alpha_i<1$$
    where the middle inequality follows from $\binom{k_0+1}{i}=\frac{\prod_{j=1}^{i}(k_0-j+2)}{i!}\le\frac{\prod_{j=1}^{i}(2k_0-j+3)}{(i+1)!}$ which in turn is a consequence of $i+1\le 2^i\le \prod_{j=1}^{i}\frac{2k_0+3-j}{k_0+2-j}$. \par
    This entails that in the lower-multiparametric model the ratio of dimensions in which the rational cohomology is not a.a.s zero is at most 2. 
\end{proof}

\begin{corollary}\label{cor:cupdim}
    For $p_i=n^{-\alpha_i}$, if the probability that $\underline{X}(n;p_1,...)$ has a non-trivial cup product over $\Q$ does not tend to $0$, then the cup product is non-trivial between $(H^k(X,\Q))^2$ and $H^{2k}(X,\Q)$.
\end{corollary}

\begin{definition}\label{expansionOp}
    For a  strongly connected $d$-\di al simplicial complex $X$, an \textbf{expansion operation} is the addition of a simplex of \di \ $d$, while keeping the component strongly connected.
    In order for this to be the case, the added simplex must share a $d-1$ \di al face with an existing face, in addition to possibly other faces.\par
    Two expansion operations $A\mapsto A\cup_Q \Delta^d, B\mapsto B\cup_W \Delta^d$ are of the same type if $Q\cong W$. 
    In particular, an expansion operation $A\mapsto A\cup_Q \Delta^d$ is called a \textbf{vertex adding operation} if $Q=\Delta^{d-1}$.
    
\end{definition} 

\begin{lemma}\label{vertBeforeAles}
For $C$ strongly connected $d$-\di al, there exists a sequence $\Delta^{d}=C_0\subset C_1\subset...\subset C_m=C$ of expansion operations, where all $C_i$ are strongly connected, and, denoting $C_{i+1}=C_i\cup_{Q_i} \Delta^{d}$, $Q_i$ contains a $d-1$-face and $C_i$ has $i+1$ maximal faces for all $i$. Such a sequence exists for any choice of $C_0$ and is called a \textbf{growing process} or \textbf{growth process} of $C$.
\end{lemma}
\begin{proof}
 Define a graph $G_C=(V_C,E_C)$, where $V_C$ is the set of maximal faces of $C$, and $(\mu,M)$ is an edge if and only if $\dim(\mu\cap M)\ge d-1$. 
    The strong connectivity of $C$ is equivalent to $G_C$ being connected, and so $G_C$ contains a spanning tree.
    Any tree has a filtration $\{v\}=T_0\subset T_1\subset ...\subset T_m=T$, where we add 1 vertex and 1 edge at a time, keeping the tree connected, and such a filtration corresponds to the desired $\Delta^d=C_0\subset C_1\subset...\subset C_m=C$. Furthermore, $C_0$ may be chosen to be any face as any vertex may be chosen to be the root of the tree.
\end{proof}

\begin{definition}\label{defPartial}
    For a simplicial complex $A$ with $a$ vertices, and a model of random simplicial complexes $X$, we denote 
    $$E(A\subset X)=\sum_{\underset{|S|=a}{S\subset V(X)}}P(A\subset X|_S),$$
    where $P(A\subset X|_S)$ is the probability of it being possible to map $A$ injectively to the induced subcomplex on $S$.
    In other words, it is the expected number of appearances of $A$ as a subcomplex of $X$. We sometimes simply write "$E(A)$" if $X$ is clear from context.
\end{definition}
In order to use the first and second moment methods, we wish to have terminology tying a growing process of a component to its expectation.
\begin{definition}\label{defBudgetCost}
     The \textbf{budget} (w.r.t \di\ $d$) is $\lim_{n\to \infty}\log_n(E(\Delta^d\subset X))$.\par
    For an expansion operation $A\mapsto A'$ the \textbf{cost} is $\lim_{n\to \infty}\log_n(\frac{E(A\subset X)}{E(A'\subset X)})$. 
    
\end{definition}
\begin{proposition}\label{propBudCostLower}
    The terms \textbf{budget} and \textbf{cost} are well defined for $\underline{X}(n;p_1,...)$ where $p_i=n^{-\alpha_i}$.
\end{proposition}
\begin{proof} 
    We prove a more general result. Assume the simplicial complex $A$ has $a_0$ vertices, $a_1$ edges and so on. We prove that $\lim_{n\to \infty}\log_n(E(A\subset \underline{X}))=a_0-\sum_{i=1} ^{\dim(A)}a_i\alpha_i$,  .\par
    For a particular choice of an injective map $f$ from the vertices of $A$ to $a_0$ vertices of $\underline{X}$, the probability of this map to extend to a map of complexes may be written as follows: 
    Define $B_k$ to be the event that $f$ sends all $k$-faces in $A$ to $k$-faces in $\underline{X}$. Then we are interested in $P\prs{\bigcap_{i=1} ^{\dim(A)}B_i}$.
    $$P\prs{\bigcap_{i=1} ^{\dim(A)}B_i}=P(B_1)P(B_2|B_1)P(B_3|B_2\cap B_1)...=\prod_{i=1} ^{\dim(A)} P\prs{B_i\middle\lvert\bigcap_{j=1} ^{i-1}B_j}$$
    by definition of conditional probability. But in the lower model $P(B_i|\bigcap_{j=1} ^{i-1}B_j)=p_i ^{a_i}=n^{-a_i\alpha_i}$. There are $\binom{n}{a_0}$ subsets of $[n]$ of size $a_0$, and for each such subset there are $a_0!$ possible maps $f$ (if $A$ has symmetries, these maps might induce the same map on complexes, but this turns out not to matter). 
    Therefore, 
    $$\binom{n}{a_0}n^{-\sum_{i=1} ^{\dim(A)}a_i\alpha_i}\le E(A\subset \underline{X})\le a_0!\binom{n}{a_0}n^{-\sum_{i=1} ^{\dim(A)}a_i\alpha_i}$$ 
    and so $\log_n(E(A\subset X))\xrightarrow{n\to \infty} a_0-\sum_{i=1} ^{\dim(A)}a_i\alpha_i$.\par
    This generalizes the calculation from \cite{Fowler19} for the simplex, giving $E(\Delta^d\subset X)=\binom{n}{d+1}n^{-\sum_{i=1} ^{d+1}\binom{d+1}{i+1}\alpha_i}$ and a budget $\log_n(E(\Delta^d\subset X))\xrightarrow{n\to \infty} d+1-\sum_{i=1} ^{d+  1}\binom{d+1}{i+1}\alpha_i$.\par
    As for expansion operations, since $\lim_{n\to \infty}\log_n(E(A\subset X))$ depends only on the number of faces of each \di, an expansion operation $A\mapsto A\cup_Q \Delta^d$ increases the number of $i$-faces by $\binom{d+1}{i+1}-q_i$, where $q_i$ is the number of $i$-faces in $Q$. So the cost is well defined, and only depends on the type of expansion operation.
\end{proof}
\begin{remark}
    A corresponding result is also true for other $p_i$, but is better stated without logarithms and will not be needed here.
    
    Note that there are in general multiple choices of a growth process for a strongly connected complex. 
\end{remark}
\begin{example}\label{expDim2}
    
 The following are the expansion operations for the pure 2-dimensional case:
\begin{center}
\begin{tabular}{| c | c | c |} 
 \hline
 Operation & Cost & Homotopical effect\\
 \hline
\begin{tikzpicture}[x=0.75pt,y=0.75pt,yscale=-1,xscale=1]

%Shape: Triangle [id:dp6994014486877365] 
\draw  [fill={rgb, 255:red, 155; green, 155; blue, 155 }  ,fill opacity=1 ] (16.25,54.25) -- (56.25,34.25) -- (56.25,74.25) -- cycle ;
%Shape: Triangle [id:dp3867613909484746] 
\draw  [color={rgb, 255:red, 208; green, 2; blue, 27 }  ,draw opacity=1 ][pattern=redstripes,pattern size=6pt,pattern thickness=0.75pt,pattern radius=0pt, pattern color={rgb, 255:red, 208; green, 2; blue, 27}] (96.25,54.25) -- (56.25,74.25) -- (56.25,34.25) -- cycle ;
%Shape: Circle [id:dp16977051056865888] 
\draw  [fill={rgb, 255:red, 208; green, 2; blue, 27 }  ,fill opacity=1 ] (92,54.25) .. controls (92,51.9) and (93.9,50) .. (96.25,50) .. controls (98.6,50) and (100.5,51.9) .. (100.5,54.25) .. controls (100.5,56.6) and (98.6,58.5) .. (96.25,58.5) .. controls (93.9,58.5) and (92,56.6) .. (92,54.25) -- cycle ;
\end{tikzpicture} &   $2\alpha_1+\alpha_2-1 >0$  & 

    None \\
 \hline 
\begin{tikzpicture}[x=0.75pt,y=0.75pt,yscale=-1,xscale=1]

%Shape: Triangle [id:dp6994014486877365] 
\draw  [fill={rgb, 255:red, 155; green, 155; blue, 155 }  ,fill opacity=1 ] (16.25,54.25) -- (56.25,34.25) -- (56.25,74.25) -- cycle ;
%Shape: Triangle [id:dp3867613909484746] 
\draw  [color={rgb, 255:red, 208; green, 2; blue, 27 }  ,draw opacity=1 ][pattern=redstripes,pattern size=6pt,pattern thickness=0.75pt,pattern radius=0pt, pattern color={rgb, 255:red, 208; green, 2; blue, 27}] (96.25,54.25) -- (56.25,74.25) -- (56.25,34.25) -- cycle ;
%Shape: Circle [id:dp16977051056865888] 
\draw  [fill={rgb, 255:red, 0; green, 0; blue, 0 }  ,fill opacity=1 ] (92,54.25) .. controls (92,51.9) and (93.9,50) .. (96.25,50) .. controls (98.6,50) and (100.5,51.9) .. (100.5,54.25) .. controls (100.5,56.6) and (98.6,58.5) .. (96.25,58.5) .. controls (93.9,58.5) and (92,56.6) .. (92,54.25) -- cycle ;
\end{tikzpicture} & $2\alpha_1+\alpha_2>1$ & Adds generator to $H^1$\\
 \hline 
\vspace*{3 pt}
\begin{tikzpicture}[x=0.75pt,y=0.75pt,yscale=-1,xscale=1,rotate=-90]

%Shape: Triangle [id:dp3867613909484746] 
\draw  [color={rgb, 255:red, 208; green, 2; blue, 27 }  ,draw opacity=1 ][pattern=redstripes,pattern size=6pt,pattern thickness=0.75pt,pattern radius=0pt, pattern color={rgb, 255:red, 208; green, 2; blue, 27}] (28.93,81.57) -- (63.57,61.57) -- (63.57,101.57) -- cycle ;
%Shape: Triangle [id:dp6994014486877365] 
\draw  [fill={rgb, 255:red, 155; green, 155; blue, 155 }  ,fill opacity=1 ] (26.25,36.93) -- (63.57,61.57) -- (28.93,81.57) -- cycle ;
%Shape: Triangle [id:dp4036264074379974] 
\draw  [fill={rgb, 255:red, 155; green, 155; blue, 155 }  ,fill opacity=1 ] (26.25,126.21) -- (28.93,81.57) -- (63.57,101.57) -- cycle ;
\end{tikzpicture} & 
    $\alpha_1+\alpha_2>\frac{1}{2}$
& 
    None
\\
 \hline 
 \begin{tikzpicture}[x=0.75pt,y=0.75pt,yscale=-0.75,xscale=0.75]

%Shape: Triangle [id:dp3867613909484746] 
\draw  [color={rgb, 255:red, 208; green, 2; blue, 27 }  ,draw opacity=1 ][pattern=redstripes,pattern size=6pt,pattern thickness=0.75pt,pattern radius=0pt, pattern color={rgb, 255:red, 208; green, 2; blue, 27}] (28.93,81.57) -- (63.57,61.57) -- (63.57,101.57) -- cycle ;
%Shape: Triangle [id:dp6994014486877365] 
\draw  [fill={rgb, 255:red, 155; green, 155; blue, 155 }  ,fill opacity=1 ] (26.25,36.93) -- (63.57,61.57) -- (28.93,81.57) -- cycle ;
%Shape: Triangle [id:dp4036264074379974] 
\draw  [fill={rgb, 255:red, 155; green, 155; blue, 155 }  ,fill opacity=1 ] (26.25,126.21) -- (28.93,81.57) -- (63.57,101.57) -- cycle ;
%Shape: Triangle [id:dp23405293669819338] 
\draw  [fill={rgb, 255:red, 155; green, 155; blue, 155 }  ,fill opacity=1 ] (103.57,81.57) -- (63.57,101.57) -- (63.57,61.57) -- cycle ;

\end{tikzpicture} & \centering $\alpha_2$ &     Reduces $H^1$ or increases $H^2$ \\
 \hline
\end{tabular}

\end{center}\par
\noindent where the costs are w.r.t the lower model and the inequalities follow from $S^1 _1(\{\alpha_i\}_i)>1$. The black and gray simplexes denote the complex before the expansion operation,  while the red patterned simplexes denote new simplexes. \par
For instance, consider the following growth process for the tetrahedron:
$$\begin{tikzpicture}

\draw[fill=gray,fill opacity=0.4]  (0,0) 
  -- (2,0) 
  -- (2.2,1) 
  -- cycle;

\draw (0,0) node[left]{$A$}
(2,0) node[right]{$B$}
  (2.2,1) node[right]{$C$};
\end{tikzpicture}
\begin{tikzpicture}

\draw[fill=gray,fill opacity=0.4]  (0,0) 
  -- (2,0) 
  -- (2.2,1) 
  -- cycle;
\draw[pattern=redstripes,pattern size=6pt,pattern thickness=0.75pt,pattern radius=0pt, pattern color={rgb, 255:red, 208; green, 2; blue, 27}]  (0,0) 
  -- (1,1.6) 
  -- (2.2,1) 
  -- cycle;
\draw (0,0) node[left]{$A$}
(1,1.6) node[above]{$D$}
(2,0) node[right]{$B$}
  (2.2,1) node[right]{$C$};
\end{tikzpicture}
\begin{tikzpicture}

\draw[fill=gray,fill opacity=0.3]  (0,0) 
  -- (2,0) 
  -- (2.2,1) 
  -- cycle;
\draw[fill=gray,fill opacity=0.3]  (0,0) 
  -- (1,1.6) 
  -- (2.2,1) 
  -- cycle;
\draw (0,0) node[left]{$A$}
(1,1.6) node[above]{$D$}
(2,0) node[right]{$B$}
  (2.2,1) node[right]{$C$};
\draw[pattern=redstripes,pattern size=6pt,pattern thickness=0.75pt,pattern radius=0pt, pattern color={rgb, 255:red, 208; green, 2; blue, 27}]
(0,0) 
  -- (1,1.6) 
  -- (2,0)
  -- cycle;
\end{tikzpicture}
\begin{tikzpicture}
\draw[fill=gray,fill opacity=0.2]  (0,0) 
  -- (1,1.6) 
  -- (2,0)
  -- cycle;

\draw[fill=gray,fill opacity=0.3]  (0,0) 
  -- (2,0) 
  -- (2.2,1) 
  -- cycle;
\draw[fill=gray,fill opacity=0.3]  (0,0) 
  -- (1,1.6) 
  -- (2.2,1) 
  -- cycle;
\draw (0,0) node[left]{$A$}
(1,1.6) node[above]{$D$}
(2,0) node[right]{$B$}
  (2.2,1) node[right]{$C$};
\draw[pattern=redstripes,pattern size=6pt,pattern thickness=0.75pt,pattern radius=0pt, pattern color={rgb, 255:red, 208; green, 2; blue, 27}] (1,1.6) 
  -- (2,0) 
  -- (2.2,1) 
  -- cycle; 
\end{tikzpicture}
$$
So we begin from the triangle $ABC$ and a budget of $3-3\alpha_1-\alpha_2$, and proceed to add vertex $D$ via the first operation, then add triangle $ABD$ via the third operation (costing $\alpha_1+\alpha_2$), and finally "close" the tetrahedron by adding triangle $BCD$ using the fourth operation (costing $\alpha_2$). 
All in all, $4-6\alpha_1-4\alpha_2$ of the budget remains, which, assuming $S^1 _1(\{\alpha_i\}_i)\geq 1$, is less than or equal to $1-\alpha_2$. That is, the expected number of subcomplexes isomorphic to the tetrahedron is $O(n^{1-\alpha_2})$.
 \end{example}
 \begin{remark}
    It is at this point that we may give an interpretation to the expressions $S^k _1(\{\alpha_i\}_i), S^k _2(\{\alpha_i\}_i)$ in the lower model. $S^k _1(\{\alpha_i\}_i)
    \geq 1$ is the requirement that the vertex adding expansion operation of \di\ $k+1$ decreases the expected number of appearances of the subcomplex.
    This means that $k+1$ \di al components are more common the smaller they are. The contribution of this condition to $k$ cohomology is that $k$-\di al cycles do not tend to get "filled" by $k+1$ faces.
    
    Meanwhile, $S^k _2(\{\alpha_i\}_i)<k+2$ is the requirement that the complex $\partial(\Delta^{k+1})$ appears as a subcomplex.
    This is the smallest $k$-cycle, as well as a prerequisite for $k+1$ faces to appear at all.
\end{remark}

\begin{proposition}\label{prop:edgeCost}
    Assume $S_1 ^k(\{\alpha_i\}_i)\geq 1$.  For a strongly connected $m$-\di al complex $C$, denote by $C'$ the result of some expansion operation adding $w$ edges and no vertices to $C$, where $m>k$. Then $E(C)= \Omega(n^{\frac{w}{k+1}} E(C'))$ .
\end{proposition} 
\begin{proof}
        If $C'=C\cup_{Q}\Delta^m$, and the edge $e\subset C'$ is not in $C$, then all faces of $\Delta^m$ containing $e$ are also not in $C$.
        Therefore, the cost of any operation adding $e$ will be higher than $\sum_{i=1} ^{m}\binom{m-1}{i-1}\alpha_i$, as we only need to choose the vertices distinct from the endpoints of the edge. We now compare coefficient-wise to $S_1^k=\sum_{i=1} ^{k+1}\binom{k+1}{i}\alpha_i\geq 1$:
        $$\frac{\prod_{j=1}^{i-1}(m-j)}{(i-1)!}=\binom{m-1}{i-1}\geq\frac{\binom{k+1}{i}}{k+1}=\frac{\prod_{j=1}^{i-1}(k+1-j)}{i!}$$
        as $m>k$ and $i\ge 1$. So,
        $$\frac{\prod_{j=1}^{i-1}(m-j)}{\prod_{j=1}^{i-1}(k+1-j)}\geq 1\geq 1/i.$$
        Thus, we see that the cost of adding an edge is at least $\frac{1}{k+1}$.\par
        Moreover, assume we are performing an operation adding $w$ edges simultaneously.
        The expansion operation is supported on a complete $(m-1)$-face and another vertex $v$, so any missing edge contains $v$ and another vertex, of which there are $w$. 
        Thus, from the remaining faces, we add at least those which contain one of these $w$, of which there are at least $\binom{m}{l}-\binom{m-w}{l}$ simplexes of \di\ $l$.
        
        We wish to show $\binom{m}{l}-\binom{m-w}{l}\geq \frac{w}{m}\binom{m}{l}$:
        $$\prod_{j=0}^{l-1}(m-j)-\prod_{j=0}^{l-1}(m-w-j)\geq w\prod_{j=1}^{l-1}(m-j),$$
        so
        $$(m-w)\prod_{j=1}^{l-1}(m-j)\geq \prod_{j=0}^{l-1}(m-w-j).$$
        But this inequality is obvious as all multiplicands on the left are greater or equal than the ones on the right. $\frac{w}{m}\binom{m}{l}\geq \frac{w}{k+1}\binom{k+1}{l}$, finishing the proof.
\end{proof}
These expansion operations should thus be thought of as "decreasing the likelihood" of a complex appearing as a subcomplex. This is actually true for all expansion operations:
\begin{proposition}\label{opDecreasesE}
    Assume $S_1 ^k(\{\alpha_i\}_i)\geq 1$ and $m> k$. For a strongly connected $m$-\di al complex $D$ denote by $D'$ the result of some expansion operation performed on $D$, adding (among others) an $l$-\di al simplex. 
    If either $S_1 ^k(\{\alpha_i\}_i)> 1$ or the operation does not add vertices and $\alpha_l >0$, then $E(D')= o(E(D))$.
\end{proposition}
\begin{proof}
    By the proof of \Cref{propBudCostLower}, $E(D)$ is of the form $n^{f_0(D)-\sum_{i=1} ^\infty f_i(D) \alpha_i}$ (up to a constant factor), where $f_i(D)$ denotes the number of $i$-\di al faces of $D$.
    If the operation added a vertex, this costs at least $S^{m-1} _1(\{\alpha_i\}_i) -1$, and since $S^{m-1} _1(\{\alpha_i\}_i)\ge S^k _1(\{\alpha_i\}_i)>1$ the result follows. 
    Otherwise, we know that $f_0(D')= f_0(D)$, $f_i(D')\ge f_i(D)$ for all $i$ and that $f_l(D')>f_l(D)$, and so $n^{f_0(D')-\sum_{i=1} ^\infty f_i(D') \alpha_i}=o(n^{f_0(D)-\sum_{i=1} ^\infty f_i(D) \alpha_i})$ as $\alpha_l >0$.
\end{proof}
\begin{corollary}    \label{finIsoTypes}
        For $\{\alpha_i\}_i$ where $S^k _1(\{\alpha_i\}_i)>1$ and $m>k$, only a finite number of isomorphism types of strongly connected components $A$ have $E(A\subset\underline{X})\not\to 0$.
\end{corollary}
    \begin{proof}
     By \Cref{opDecreasesE}, all expansion operations decrease the asymptotics of $E(\bullet\subset\underline{X})$.
     In particular, if we find a number $v_0$ such that any $m$-\di al component $A$ with $|A|>v_0$ has $E(A\subset\underline{X})\to 0$, the claim will follow.

    Denote the budget by $b$, and the cost of a vertex adding operation by $c_m:=\sum_{i=1}^m\binom{m}{i}\alpha_i-1\ge S_1^k-1>0$. If $A$ has $a$ vertices, $E(A\subset\underline{X})\sim n^{b-(a-m-1)c_m}$. For $c_m>\frac{b}{a-m-1}$ the exponent will be negative, so the expectation will go to 0.

    Furthermore, the probability that any component has more vertices than $\frac{b}{c_m}+m+1$ is bounded by the sum of probabilities of the components with exactly $\ceil{\frac{b}{c_m}+m+1}$ vertices (which is finite), and still tends to 0.
    \end{proof}

\section{Cup Product for $k=1$} \label{cup1}
\begin{proof}[Proof of \Cref{thmcup} for $k=1$]
    By \Cref{cor:cupdim} it suffices to examine the possible strong connectivity components of \di \ $2k=2$.
    There are $4$ possible expansion operations $A\mapsto A\cup_Q \Delta^2$ in \di\ 2 (illustrated in \Cref{expDim2}): 
    \begin{enumerate}[(a)]
        \item Adding a vertex, i.e $Q=\Delta^1$.
        \item Connecting a vertex and an unrelated edge, i.e $Q=\Delta^1\sqcup \Delta^0$.
        \item A "horn-filling"- $Q=\Delta^1\cup_{\Delta^0}\Delta^1$.
        \item A "sphere filling"- $Q=\partial(\Delta^2)$.
    \end{enumerate}

    By \Cref{finIsoTypes} we only have to prove that a finite collection of possible strongly connected 2-\di al components have a trivial cup product.
    The operations composing their growing processes have the following homotopical effects:
 Operations (a) and (c) induce homotopy equivalences. Operation (d) is akin to gluing a 2-cell, and thus cannot increase the number of generators of $\pi_1$. 
    Operation (b) is akin to attaching a 1-cell, named $e_1$ as a path. This can only increase the minimal number of generators of $\pi_1$ by 1, because of van Kampen's theorem: cover the new complex by everything but the middle of $e_1$, and a circle containing $e_1$.
    \par Assume a strongly connected 2-complex $A$ requires at least $\beta$ of operation (b) to construct. Then the expected number of times $A$ appears in $X$ is no greater than $n^{2-\beta(2\alpha_1+\alpha_2)}=o(n^{2-\beta})$ since $S_1^1=2\alpha_1+\alpha_2>1$ in our range. 
    Therefore, only complexes where $\beta=0,1$ can appear. This means that $\pi_1(A)$, and thus $H_1(A)$, have at most 1 generator. By the universal coefficient theorem, $H_1(A)\simeq H^1(A)$, since $H_0(A)$ is always free, and thus $H^1(A)$ also has at most a single generator.
    \begin{remark}
        Since for $c_1,c_2\in H^1$, $c_1\cup c_2=(-1)^{1\cdot 1}c_2\cup c_1=-c_2\cup c_1$, this finishes the proof for rings where $2$ is not a zero-divisor.
    \end{remark}
    \par To have $H^1\neq 0$ one needs a non-trivial $\pi_1$, so we may assume $\beta=1$. Define $v\in A$ to be an \textbf{internal vertex} if $H^1(lk(v))\ne 0$, and \textbf{external} otherwise. In particular, an internal $v$ is contained in at least 3 triangles.
    During a growing process, a vertex may become internal as a result of applying operation (c) or (d). Operation (c) costs $n^{-\alpha_1-\alpha_2}<n^{-1/2}$, and thus may only happen once (in addition to operation (b)). 
    Operation (d) does not have such a limitation, however the last edge of the triangle along which (d) is attached has to be added by either operation (c) or (b), since we are adding an edge but not a vertex.
    If it is operation (b), 
% Pattern Info
    then any path between the unpatterned triangles in \Cref{fig:proofk1} is homotopic to a path outside the patterned triangles. Since the complex was simply connected before attaching the patterned triangles, it returns to being so after their addition, and thus the strong connectivity component will have a trivial $\pi_1$ and thus no cup product.

    Otherwise, we may only have 4 internal vertices- 1 from operation (c) and 3 from the single application of operation (d) it enables. In addition there might be 1 vertex with a disconnected link (created by operation (b)). 
         \usetikzlibrary{patterns}
\tikzset{
pattern size/.store in=\mcSize, 
pattern size = 5pt,
pattern thickness/.store in=\mcThickness, 
pattern thickness = 0.3pt,
pattern radius/.store in=\mcRadius, 
pattern radius = 1pt}
\makeatletter
\pgfutil@ifundefined{pgf@pattern@name@_e5xr52g8o}{
\pgfdeclarepatternformonly[\mcThickness,\mcSize]{_e5xr52g8o}
{\pgfqpoint{0pt}{-\mcThickness}}
{\pgfpoint{\mcSize}{\mcSize}}
{\pgfpoint{\mcSize}{\mcSize}}
{
\pgfsetcolor{\tikz@pattern@color}
\pgfsetlinewidth{\mcThickness}
\pgfpathmoveto{\pgfqpoint{0pt}{\mcSize}}
\pgfpathlineto{\pgfpoint{\mcSize+\mcThickness}{-\mcThickness}}
\pgfusepath{stroke}
}}
\makeatother
%dfghdcf
\tikzset{
pattern size/.store in=\mcSize, 
pattern size = 5pt,
pattern thickness/.store in=\mcThickness, 
pattern thickness = 0.3pt,
pattern radius/.store in=\mcRadius, 
pattern radius = 1pt}
\makeatletter
\pgfutil@ifundefined{pgf@pattern@name@_7la7libyc}{
\pgfdeclarepatternformonly[\mcThickness,\mcSize]{_7la7libyc}
{\pgfqpoint{0pt}{0pt}}
{\pgfpoint{\mcSize+\mcThickness}{\mcSize+\mcThickness}}
{\pgfpoint{\mcSize}{\mcSize}}
{
\pgfsetcolor{\tikz@pattern@color}
\pgfsetlinewidth{\mcThickness}
\pgfpathmoveto{\pgfqpoint{0pt}{0pt}}
\pgfpathlineto{\pgfpoint{\mcSize+\mcThickness}{\mcSize+\mcThickness}}
\pgfusepath{stroke}
}}
\makeatother
\tikzset{every picture/.style={line width=0.75pt}} %set default line width to 0.75pt        
\begin{figure}[h]
    \centering
$$\begin{tikzpicture}[x=0.75pt,y=0.75pt,yscale=-1,xscale=1]
%Shape: Polygon [id:ds02116723121272135] 
\draw   (391,105) -- (390,206) -- (340,156) -- cycle ;
%Shape: Polygon [id:ds7977611498618582] 
\draw   (237.05,105.99) -- (235,207) -- (185,157) -- cycle ;
%Shape: Polygon [id:ds7023139050387472] 
\draw  [pattern=_e5xr52g8o,pattern size=6pt,pattern thickness=0.75pt,pattern radius=0pt, pattern color={rgb, 255:red, 255; green, 0; blue, 0}] (235.65,207) -- (237.05,105.99) -- (340,156) -- cycle ;
%Shape: Polygon [id:ds9879723815789456] 
\draw  [pattern=_7la7libyc,pattern size=6pt,pattern thickness=0.75pt,pattern radius=0pt, pattern color={rgb, 255:red, 19; green, 212; blue, 19}] (237.05,105.99) -- (391,105) -- (340,156) -- cycle ;

\end{tikzpicture}
$$
    \caption{The top green patterned triangle was added by operation (d) (the last edge existing outside the diagram), while the middle red patterned triangle was added by (b).}
    \label{fig:proofk1}
\end{figure}
    \begin{lemma}
        Let $X$ be a 2-dimensional strongly connected complex, and assume $v$ is an external vertex of $X$. Then $X$ has a non-trivial cup product if and only if $X \setminus v$ has a non-trivial cup product.
    \end{lemma}
    \begin{proof}
    The link of an external vertex in a 2-dimensional complex is a cycle-free graph, so it may be either contractible or a disjoint union of contractible components.
    Erasing a vertex with a contractible link (and all faces containing it) does not affect the homotopy type of the complex.
    This is because any vertex is a cone-point above its link, and if the link is contractible then it may safely be erased along with all simplexes containing it.\par
    If a vertex $v$ has a disconnected link, let $A_i$ be the connected components of $lk(v)$. Define a new complex $X'=((X\setminus v)\cup_{A_1}C(A_1)) \cup_{A_2}C(A_2)...$, where $C(A)$ is the cone over $A$. Denote by $\crb{v_i}$ the new vertices (the cone-points).
    Define $X''=X'\cup\bigcup_{i}(v_i,v_{i+1})$. The mapping $X''\to X$, that sends all $v_i$ to $v$ and is the identity on $X\setminus v$, is clearly an equivalence.
    But $X''\sim X'\vee \bigvee_{i=1}^l S^1$ for some $l$, and these $S^1$ are not part of the strongly connected 2-\di al component, so do not affect the cup product. Thus, it is enough to look at $X'$, where all $v_i$ have contractible links, and by the previous claim it is enough to look at $X\setminus v$. 
    
    See \Cref{seperationExp} for an example of this process.\end{proof}
    
\begin{figure}[h]
    \centering
    \tikzset{every picture/.style={line width=0.75pt}} %set default line width to 0.75pt        

\begin{tikzpicture}[x=0.75pt,y=0.75pt,yscale=-0.5,xscale=0.5]

%Shape: Triangle [id:dp736164700426237] 
\draw  [fill={rgb, 255:red, 155; green, 155; blue, 155 }  ,fill opacity=1 ] (147,129) -- (182,242) -- (112,242) -- cycle ;
%Shape: Triangle [id:dp775139456327878] 
\draw  [fill={rgb, 255:red, 155; green, 155; blue, 155 }  ,fill opacity=1 ] (147,129) -- (227.36,42.19) -- (262.36,102.81) -- cycle ;
%Shape: Triangle [id:dp5945617461728145] 
\draw  [fill={rgb, 255:red, 155; green, 155; blue, 155 }  ,fill opacity=1 ] (147,129) -- (31.64,102.81) -- (66.64,42.19) -- cycle ;
%Shape: Triangle [id:dp9381310534217522] 
\draw  [fill={rgb, 255:red, 155; green, 155; blue, 155 }  ,fill opacity=1 ] (497,153) -- (532,266) -- (462,266) -- cycle ;
%Shape: Triangle [id:dp17061599434635832] 
\draw  [fill={rgb, 255:red, 155; green, 155; blue, 155 }  ,fill opacity=1 ] (525,107) -- (605.36,20.19) -- (640.36,80.81) -- cycle ;
%Shape: Triangle [id:dp7442179196538858] 
\draw  [fill={rgb, 255:red, 155; green, 155; blue, 155 }  ,fill opacity=1 ] (470,107) -- (354.64,80.81) -- (389.64,20.19) -- cycle ;
%Shape: Circle [id:dp3604237107030763] 
\draw  [fill={rgb, 255:red, 0; green, 0; blue, 0 }  ,fill opacity=1 ] (141.5,129) .. controls (141.5,125.96) and (143.96,123.5) .. (147,123.5) .. controls (150.04,123.5) and (152.5,125.96) .. (152.5,129) .. controls (152.5,132.04) and (150.04,134.5) .. (147,134.5) .. controls (143.96,134.5) and (141.5,132.04) .. (141.5,129) -- cycle ;
%Shape: Circle [id:dp49374764318809583] 
\draw  [fill={rgb, 255:red, 0; green, 0; blue, 0 }  ,fill opacity=1 ] (519.5,107) .. controls (519.5,103.96) and (521.96,101.5) .. (525,101.5) .. controls (528.04,101.5) and (530.5,103.96) .. (530.5,107) .. controls (530.5,110.04) and (528.04,112.5) .. (525,112.5) .. controls (521.96,112.5) and (519.5,110.04) .. (519.5,107) -- cycle ;
%Shape: Circle [id:dp6893090157158459] 
\draw  [fill={rgb, 255:red, 0; green, 0; blue, 0 }  ,fill opacity=1 ] (464.5,107) .. controls (464.5,103.96) and (466.96,101.5) .. (470,101.5) .. controls (473.04,101.5) and (475.5,103.96) .. (475.5,107) .. controls (475.5,110.04) and (473.04,112.5) .. (470,112.5) .. controls (466.96,112.5) and (464.5,110.04) .. (464.5,107) -- cycle ;
%Shape: Circle [id:dp5103763322156964] 

%Shape: Circle [id:dp10463331229822637] 
\draw  [fill={rgb, 255:red, 0; green, 0; blue, 0 }  ,fill opacity=1 ] (491.5,153) .. controls (491.5,149.96) and (493.96,147.5) .. (497,147.5) .. controls (500.04,147.5) and (502.5,149.96) .. (502.5,153) .. controls (502.5,156.04) and (500.04,158.5) .. (497,158.5) .. controls (493.96,158.5) and (491.5,156.04) .. (491.5,153) -- cycle ;
%Straight Lines [id:da6475651227065742] 
\draw    (470,107) -- (497,153) ;
%Straight Lines [id:da22993599914065488] 
%Straight Lines [id:da5277291647029594] 
\draw    (525,107) -- (497,153) ;
\end{tikzpicture}
    \caption{An example of the process of separation.}
    \label{seperationExp}
\end{figure}

 All in all, we may restrict to the subcomplex induced on the internal vertices, of which we have at most 4. This may only be (a boundary of) a tetrahedron, which plainly does not have a cup product in any coefficient ring.
\end{proof} 

\section{Cup Product for higher $k$}\label{cupK}
By \Cref{cor:cupdim}, it is enough to examine strong connectivity components of \di\ $2k$, where $k$ is the minimal number s.t $S_1 ^k(\{\alpha_i\}_i)\geq 1$.
\begin{proposition}\label{vertAmount}
    For $k\ge 2$, assume $p_i=n^{-\alpha_i}, \alpha_i>0$, such that $S_1 ^k(\{\alpha_i\}_i)\geq 1$ and $S_2 ^{2k}(\{\alpha_i\}_i)<2k+2$. Then, a.a.s no $2k$-\di al strong connectivity component has more than $2k + 1 + a(k)$ vertices, where $a(2)=4, a(3)=3$ and $ a(k)=2,k\geq 4$. 
\end{proposition}
\begin{proof}
    We know that $S_1 ^k=\sum_{i=1} ^{k+1}\binom{k+1}{i}\alpha_i\ge 1$ and $S_2^{2k}=\sum_{i=1} ^{2k}\binom{2k+2}{i+1}\alpha_i< 2k+2$. A strongly-connected component of \di \ $2k$ starts with a budget (see \Cref{defBudgetCost}) of 
    $$\lim_{n\to \infty}\log_n(E(\Delta^{2k}\subset X))=2k+1-\sum_{i=1} ^{2k}\binom{2k+1}{i+1}\alpha_i.$$
    Using $\sum_{i=1} ^{k+1}\binom{k+1}{i}\alpha_i\geq 1$ and $k\geq 2$ we see that:
        $$\sum_{i=1} ^{2k}\binom{2k+1}{i+1}\alpha_i=\sum_{i=1} ^{2k}\frac{\prod_{j=0}^i(2k+1-j)}{(i+1)!}\alpha_i\ge $$$$\sum_{i=1} ^{k+1}\frac{(2k+1)(2k+1-i)2^{i-1}\prod_{j=0}^{i-1}(k+1-j)}{(i+1)!(k+1)}\alpha_i+\sum_{i=k+2} ^{2k}\binom{2k+1}{i+1}\alpha_i\ge
        $$$$\frac{\binom{2k+1}{2}}{k+1}\prs{\sum_{i=1} ^{k+1}\binom{k+1}{i}\alpha_i}+\sum_{i=k+2} ^{2k+1}\binom{2k+1}{i+1}\alpha_i> \frac{\binom{2k+1}{2}}{k+1}=\frac{(2k+1)k}{k+1}=2k+1-2+\frac{1}{k+1}.$$
        This means that our initial budget is less than $2-\frac{1}{k+1}$. 
        Additionally, the cost of adding a vertex is $1-\sum_{i=1} ^{2k}\binom{2k}{i}\alpha_i$. By the same method, we know that 
        $$\sum_{i=1} ^{2k}\binom{2k}{i}\alpha_i\geq \frac{2k}{k+1}\prs{\sum_{i=1} ^{k+1}\binom{k+1}{i}\alpha_i}\geq \frac{2k}{k+1}=2-\frac{2}{k+1}$$
        thus, adding a vertex costs at least $-1+(2-\frac{2}{k+1})=1-\frac{2}{k+1}$.
        This, in turn, means that the number of additional vertices (not present in the initial simplex) a component can have is very restricted when $k\geq 2$. Specifically, for $k=2$ it may have at most 4 additional vertices, for $k=3$ only 3, and for $k\geq 4$ only 2.
\end{proof}
    
\begin{proof}[Proof of \Cref{thmcup} for $k\geq 2$]
    We have bounded the possible numbers of vertices in strongly connected $2k$-components in \Cref{vertAmount}, and now only need to address the small number of possible topologies that may be supported on so few vertices. Let $C$ be a $2k$-\di al strongly connected component. Dividing into cases by the number of vertices $C$ has over $2k+1$:
    \begin{enumerate}
        \item[\textbf{0,1}:] Here the only possible homotopy types are $*,S^{2k}$, both of which have a cup length of 1.

        \item[\textbf{2}:] If 2 vertices were added, the remaining budget is less than
        $$2-\frac{1}{k+1}+2(\frac{2}{k+1}-1)=\frac{3}{k+1}.$$
        By \Cref{prop:edgeCost}, the cost of an edge is at least $\frac{1}{k+1}$, thus in a growing process of $C$ at most 2 edges are added by operations not adding vertices.
        Any $2k$-\di al complex on $2k+3$ vertices has at most $\binom{2k+3}{2}=2k^2+5k+3$ edges. 
        Since we may add only 2 edges in addition to the ones added while adding vertices, and a vertex adding operation increases the number of edges by $2k$, we have at most $\binom{2k+1}{2}+4k+2=\frac{4k^2+2k+8k+4}{2}=2k^2+5k+2$, which is 1 less than the complete complex. 
        
        This means that $C$ has a pair of vertices that are not connected by an edge. Denote them by $v,w$.
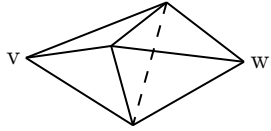
            \begin{center}
        \begin{tikzpicture}[x=0.75pt,y=0.75pt,yscale=-1,xscale=1]

%Shape: Circle [id:dp6314681860503732] 

%Straight Lines [id:da2925002954692024] 
\draw    (91.25,71.25) -- (48.25,77.25) ;
%Straight Lines [id:da9062146583360271] 
\draw    (48.25,77.25) -- (102.25,111.25) ;
%Straight Lines [id:da6527690291475796] 
\draw    (119.25,49.25) -- (48.25,77.25) ;
%Straight Lines [id:da33215228750422376] 
\draw    (158.25,79.25) -- (102.25,111.25) ;
%Straight Lines [id:da2914456838455004] 
\draw    (119.25,49.25) -- (158.25,79.25) ;
%Straight Lines [id:da1590258206295112] 
\draw    (119.25,49.25) -- (91.25,71.25) ;
%Straight Lines [id:da24964285934972397] 
\draw    (91.25,71.25) -- (102.25,111.25) ;
%Straight Lines [id:da1972378210348893] 
\draw  [dash pattern={on 4.5pt off 4.5pt}]  (119.25,49.25) -- (102.25,111.25) ;
%Straight Lines [id:da7226844630280271] 
\draw    (158.25,79.25) -- (91.25,71.25) ;
% Text Node
\draw (37,73) node [anchor=north west][inner sep=0.75pt]   [align=left] {v};
% Text Node
\draw (160,75) node [anchor=north west][inner sep=0.75pt]   [align=left] {w};
\end{tikzpicture}
    \end{center}
        
        $lk(v)$ is a pure $2k-1$-\di al complex that has $2k+1$ (or $2k$)  vertices, thus it is either the $(2k-1)$-\di al sphere or contractible. If it is contractible, erasing $v$ preserves the homotopy type.
        Similarly for $w$.
        If one of $v,w$ was erased- we return to the case where at most 1 vertex was added. Otherwise, $lk(v)=lk(w)=\partial(\Delta^{2k})$, and the only remaining question is whether the $2k+1$ vertices other than $v,w$ form a simplex. If they do, the component is homotopic to a wedge of 2 spheres. Otherwise, it is homotopic to a sphere.
        In both cases, there is no cup product.
        
        \item[\textbf{3}:] In this case the remaining budget is $2-\frac{1}{k+1}+3(\frac{2}{k+1}-1)=\frac{5}{k+1}-1$, thus this case is only relevant for $k=2,3$, for which the remaining budget is less than $\frac{2}{3},\frac{1}{4}$ respectively. This means that at most 1 edge can be added (and 0 if $k=3$) besides the vertex adding operations. Denote this edge by $e$ and its endpoints by $x,y$.
        
        Any expansion operation besides adding vertices or $e$ must be supported on $2k+1$ vertices whose 1-skeleton forms a clique.
        This is only possible if this clique includes $e$, otherwise we are adding an existing simplex. 
        
        Therefore, all operations apart from vertex additions must happen after adding $e$. Furthermore, we may assume the vertex adding operations occurred before adding $e$. 
        This is because if $w$ was the last vertex added after $e$, $lk(w)$ in $C$ must be $\Delta^{2k-1}$, as no later operation can include $w$.
        $w$ therefore can be erased while preserving strong connectivity, returning us to the previous case.
        
        If any additional operation must include $x,y$ and only their common neighbors, how many of those exist? Assume WLOG that in the growth process $C_0\subset C_1\subset...\subset C$, $y$ was added after $x$ and appears first in $C_y$. 
        $x$ and $y$ may have at most $2k$ common neighbors in $C_y$ (as this is the number of neighbors of $y$). 
        Any vertex $v$ added later may not be a common neighbor, since the operation $C_j\cup_Q\Delta^{2k}$ adding $v$ must have $Q=\Delta^{2k-1}$, which before adding $e$ cannot contain both $x,y$.
        
        Thus, all operations after the addition of $e$ will only be attached on a total of $2k+2$ vertices ($x,y$ and their common neighbors). The remaining 2 vertices will have contractible link, and thus can be erased without changing the homotopy type. 
        This new complex, $C''$ is either $\partial(\Delta^{2k+1})$ or has $H^{2k}(C'')=0$. At any rate, $C''$ has cup length 1, and therefore so does $C$.
        \item[\textbf{4}:] This is only relevant for $k=2$, where we remain with a budget of less than $\frac{1}{3}$. Thus, no edge can be added, so only vertex adding operations occur in the growing process. Therefore, the component is contractible.\qedhere
    \end{enumerate}
\end{proof}
\begin{remark}
    To avoid confusion it is important to stress that the proposition only addresses strongly connected $\textbf{2k}$-components. 
    The proof that $H^l\neq 0$ for (in our case) $k\le l<2k$ in \cite{Fowler19} involves proving the existence of $\partial\Delta^l$ as a subcomplex of $X$ with a maximal face. 
    This maximal face is by definition not contained in any strongly connected component of dimension greater than $l$.
\end{remark}

\section{Steenrod Squares and Powers}\label{steenChap}
\begin{proof}[Proof of \Cref{thmSteenrod}]
\begin{lemma}\label{constCSigma}
    Consider $\sigma\in A_p$, the Steenrod algebra for the prime $p$. Then there exists a strongly connected pure-dimensional simplicial complex $C_\sigma$ with $\sigma:H^*(C_\sigma;\Z/p)\to H^{*+\deg(\sigma)}(C_\sigma;\Z/p)$ a non-trivial map.
\end{lemma}
\begin{proof}
     From \cite[p.500]{AT} we know that there exists $c\in \N_{\ge 0}$ such that for $\iota$, a generator of $H^c(K(\Z/p,c);\Z/p)$, 
     $\sigma(\iota)$ is non-trivial in $H^*(K(\Z/p,c);\Z/p)$.
     Furthermore, we know from \cite[Prop. 4.10.1]{BERGER07} that $K(\Z/p,c)$ has a dimension-wise finite CW structure. 
     Therefore, by taking the $c+\deg(\sigma)$ skeleton of this structure, we would get a finite CW-complex with $\sigma$ acting non-trivially on cohomology (as both $\iota$ and $\sigma(\iota)$ would be unaffected by the truncation).
     \cite[Theorem 2C.5]{AT} says that every finite CW complex is homotopy equivalent to a finite simplicial complex, and by \Cref{lemCupSteenConc} we may restrict ourselves to a strongly connected component on which $\sigma$ acts non-trivially, granting the result.
\end{proof}

There is no reason to believe that $C_\sigma$ appears as a subcomplex of $\underline{X}(n;n^{-\alpha_1},...)$ with any non-negligible probability. However, the next construction allows us to modify the complex such that it does:
  
\begin{construction}\label{conKoganSus}
    Let $C$ be a finite (strongly connected) simplicial complex. Define a new complex $C'$ as follows: add a vertex 
    $v$ as a cone point above $C$. Then add the simplex supported on all vertices of $C$.
    This results in a complex $C'$ homotopy equivalent to $\Sigma C$ by the following map: Let $\Sigma C$ be modeled as $(C\times [0,1]/C\times \{1\})/C\times \{0\}$. Then the map that sends $v$ to $1$ and all simplexes supported on $C$ to 0 is easily seen to be a homotopy equivalence.\par
    The above construction, however, does not preserve purity of dimension. To achieve that, assuming $\dim(C)=d$, define $\Sigma' C=sk_{d+1}(C')$. This is seen to preserve both purity of dimension and strong connectivity. 
    Moreover, for any $l\le d$, if $Im(\sigma)\subset H^l(C)$ is non-trivial, for $\sigma\in A_p$, then $Im(\sigma)\subset H^{l+1}(\Sigma'C)$ is non-trivial as well.
    This is since suspension preserves them, and erasing the higher dimensional simplexes only expands the set of cocycles in relevant \di s, thus satisfying the requirements in the corollary  to \Cref{lemCupSteenConc}. \par
    Notice that if $C$ has a complete $l$-skeleton, $\Sigma'C$ has a complete $l+1$-skeleton. In particular, $\Sigma'C$ always has a complete 1-skeleton if $\dim(C)\geq 1$.
\end{construction}

We now apply the construction to $C_\sigma$:
\begin{lemma}\label{lem:susGood}
    For any $C_\sigma$ as in \Cref{constCSigma}, there exist $\{\alpha_i\}_i$ and $r\in\N$ such that $E(\Sigma'^r C_\sigma\subset\underline{X}(n;n^{-\alpha_1},...))\to \infty$.
\end{lemma}

\begin{proof}
Assume $C_\sigma$ has $v$ vertices, $\dim(C_\sigma)=d$, and minimal non-trivial cohomology in \di \  $k$. Then $\Sigma'C_\sigma$ has $v+1$ vertices and, as we saw, a complete 1-skeleton. 
This means that $E(\Sigma' C_\sigma\subset\underline{X})\sim n^{v+1-\binom{v+1}{2}\alpha_1-\sum_{i=2} ^{d+1} a_i\alpha_i}$, where $a_i\in \N$. 
$$v+1-\binom{v+1}{2}\alpha_1-\sum_{i=2} ^{d+1} a_i\alpha_i=v+1-\frac{\binom{v+1}{2}}{k+2}S_1^{k+1}-\sum_{i=2} ^{d+1} b_{1,i}\alpha_i$$
and similarly
$$\log_nE(\Sigma'^r C_\sigma\subset\underline{X})\sim v+r-\binom{v+r}{2}\alpha_1-\sum_{i=2} ^{d+r+1} a_{r,i}\alpha_i=v+r-\frac{\binom{v+r}{2}}{k+r+1}S^{k+r}_1-\sum_{i=2} ^{d+r+1} b_{r,i}\alpha_i$$
for some $\crb{b_{r,i}}$. If $S_1^{k+r}\ge 1$ is very close to $1$ and $\crb{\alpha_i},i\ge 2$ are sufficiently close to $0$ (which is possible by taking $\alpha_1\approx\frac{1}{k+r+1}$), the expression will be positive if $v+r>\frac{\binom{v+r}{2}}{k+r+1}$, or equivalently $r>v-2k-3$.

 Thus, $E(\Sigma'^r C_\sigma\subset \underline{X})$ tends to infinity for a certain choice of $\alpha$'s, which we assume from now on.
 The conditions "$S_1^{k+r}\ge 1$ is very close to $1$ and $\crb{\alpha_i},i\ge 2$ are sufficiently close to $0$" translates to a  set of choices with non-empty interior for the $\alpha$'s, as required by \Cref{thmSteenrod}.
\end{proof}

\begin{lemma}\label{secondMom}
    For $d,v,r,k$ as above, denote $M_n:=\sum_{i\in \binom{[n]}{v+r}}\mathbb{1}_i$, where $\mathbb{1}_i$ is the indicator of whether the induced ordered complex on the vertices contains $\Sigma'^r C_\sigma$. For $\{\alpha_i\}_i$ such that $E(M_n)\to \infty$ and $S_1 ^{k+r}(\{\alpha_i\}_i)\geq 1$, $\underline{X}$ a.a.s contains more than $\frac{E(M_n)}{2}$ copies of $\Sigma'^r C_\sigma$.
\end{lemma}
\begin{proof}
        We use the second moment method. 
    By Chebyshov's inequality $$P\prs{|M_n-E(M_n)|\geq \frac{E(M_n)}{2}}\leq 4\frac{Var(M_n) }{E^2(M_n)}$$ 
    In other words, we wish to show that $Var(M_n)=o(E^2(M_n))$.  
    $$Var(M_n)=\sum_{i,j \in \binom{[n]}{v+r}} (E(\mathbb{1}_i\mathbb{1}_j)-E(\mathbb{1}_i)E(\mathbb{1}_j))=\sum_{i,j \in \binom{[n]}{v+r}}Cov(\mathbb{1}_i,\mathbb{1}_j),$$
    so we prove the claim by analyzing the summands based on the size of the intersection $|i\cap j|=:m$:
   
    \begin{itemize}
        \item $Cov(\mathbb{1}_i,\mathbb{1}_j)=0$ when $|i\cap j|=0,1$, since there is no simplex of positive \di \  common to both.
        \item If $2\leq m\leq d+r+1$, then $\sum_{|i\cap j|=m} E(\mathbb{1}_i\mathbb{1}_j)$ 
        %(we address $E(\mathbb{1}_i)E(\mathbb{1}_j)$ later) 
        is equal to (up to a scalar bounded by $\binom{v+d+1}{m}^2$) 
        $$\sum_{W\subset \Sigma'^r C} E(\Sigma'^r C_\sigma\cup_W \Sigma'^r C_\sigma\subset \underline{X}),|W|=m.$$
        Denoting $E(\mathbb{1}_j):=\Theta(n^{-c})$, by \Cref{propBudCostLower} $\sum_{|i\cap j|=m} E(\mathbb{1}_i\mathbb{1}_j)=\Theta(n^{2(v+r)-m-2c+q_m})$, where $q_m\leq \sum_{l=1} ^{m-1} \binom{m}{l+1}\alpha_l$.
        But by \di ality 
        $$n^{m-\sum_{l=1} ^{m-1} \binom{m}{l+1}\alpha_l}\sim E(\Delta^{m-1}\subset\underline{X})=\omega(E(\Sigma'^r C_\sigma\subset\underline{X}))=\omega(1),$$
        and we know that $\sum_{l=1} ^{m-1} \binom{m}{l+1}\alpha_l-m<0$, thus $n^{-2c+q_m-m}=o(n^{2(v+r)-2c})$. 
        The number of isomorphism types of such intersections does not depend on $n$, and so this part of the sum is $o(n^{2(v+r)-2c})=o(E(M_n)^2)$.

        \item If $d+r+2\leq m< v+r$, then $\mathbb{1}_i\cdot\mathbb{1}_j$ is an indicator of a strongly connected complex as $\Sigma'^r C_\sigma$ has, apart from the last vertex, a complete $d+r$ skeleton. 
        Such a complex forms an expansion of $\Sigma'^r C_\sigma$, so by \Cref{opDecreasesE} ($S_1^{d+r}>S_1^{k+r}$) and the fact that $S_1 ^{k+r}(\{\alpha_i\}_i)\geq 1$, $\sum_{|i\cap j|=m}E(\mathbb{1}_i\mathbb{1}_j)$ is (up to a constant) a sum over all isomorphism types where each summand is $o(n^{v+r-c})$. The number of summands is independent of $n$, so $\sum_{|i\cap j|=m}E(\mathbb{1}_i\mathbb{1}_j)=o(E(M_n))$.
        \item In the case where the intersection is full, we have $E(M_n)$ as the sum, which is certainly $o(E(M_n)^2)$.
        \item $E(\mathbb{1}_i)E(\mathbb{1}_j)$ where $m\geq 2$ may be ignored since their contribution is $\Theta(n^{2v+2r-2-2c})=o(E(M_n)^2)$.
    \end{itemize}
    All in all, we have that $$P\prs{|M_n-E(M_n)|\geq \frac{E(M_n)}{2}}\leq 4\frac{Var(M_n) }{E^2(M_n)}\xrightarrow{n\to \infty} 0$$ and we get the result. 
    \end{proof}
    
    \begin{lemma}\label{lem:maxComponents}
        For positive  $\{\alpha_i\}_i$ such that $E(M_n)\to \infty$ and $S_1 ^{k+r}(\{\alpha_i\}_i)> 1$, $\underline{X}$ a.a.s contains a subcomplex isomorphic to $\Sigma'^r C_\sigma$, whose $k+r+1$ skeleton is a strongly connected component.
    \end{lemma}
    
    \begin{proof}
    As we saw in \Cref{secondMom}, $\underline{X}$ a.a.s contains $\Theta(E(M_n))$ copies of $\Sigma'^r C_\sigma$ as subcomplexes. Let $D$ be a complex whose $k+r+1$ skeleton is strongly connected, containing $\Sigma'^r C_\sigma$ as a proper subcomplex. 
    Denote by $N_n:=\sum_{j\in \binom{[n]}{|D|}}\mathbb{1}_j$, where $\mathbb{1}_j$ is the indicator of whether the induced ordered complex on $j$ contains $D$. Since $S^{k+r} _1(\{\alpha_i\}_i)> 1$, $E(N_n)=o(E(M_n))$ by \Cref{opDecreasesE}. Thus, by Markov's inequality, 
    $$P(N_n\geq E(M_n)/m)\leq \frac{mE(N_n)}{E(M_n)}\xrightarrow[n\to \infty]{}0$$
     for any $m>0$. This is true for any such $D$, so in particular for any $D$ which, in addition to $\Sigma'^r C_\sigma$, has at most one simplex of \di\ greater than $k+r$. There is a finite number of such complexes, so by \Cref{secondMom} there a.a.s must exist a copy of $\Sigma'^r C_\sigma$ whose $k+r+1$ skeleton is a strongly connected component.
    \end{proof}
    Thus, for any cocycle $c\in H^{k+r}(\Sigma'^r C_\sigma)$ and any copy of $\Sigma'^r C_\sigma$  whose $k+r+1$ skeleton is a strongly connected component in $\underline{X}$, both $c$ and $\sigma (c)$ may be extended to cocycles in $\underline{X}$ (by extending them as 0 to the rest of the complex), and by applying \Cref{lemCupSteenConc} we get the result.
\end{proof}
\begin{example}
    For  $\R P^2$-the projective plane- it is known that $\beta=Sq^1:H^1(\R P^2;\Z/2)\to H^2(\R P^2;\Z/2)$ is a non-zero map.
    
    In \Cref{fig:RP2} we see the minimal triangulation of $\R P^2$, obtained by quotienting the icosahedron by the antipodal map. Denoting this complex by $P$, it has 6 vertices, 15 edges and 10 2-faces, so $E(P\subset\underline{X})\sim n^{6-15\alpha_1-10\alpha_2}$.

    In order for $P$ to contribute to $H^1$, we must assume that $S_1^1>1$ as we did in \Cref{lem:maxComponents}. $S_1^1=2\alpha_1+\alpha_2$, so    $6-15\alpha_1-10\alpha_2=6-6S_1^1-3\alpha_1-4\alpha_2<0$, thus $P$ does not appear as a component in $\underline{X}$.

    Following \Cref{conKoganSus}, $\Sigma'P$ has 7 vertices, $\binom{6}{2}+6=21$ edges, $\binom{6}{3}+15=35$ 2-faces and $\binom{6}{4}+10=25$ 3-faces. $S_1^2=3\alpha_1+3\alpha_2+\alpha_3$, and so 
    $$E(\Sigma'P\subset\underline{X})\sim n^{7-21\alpha_1-35\alpha_2-25\alpha_3}=n^{7-7S_1^2 -14\alpha_2-18\alpha_3}\to 0$$
    if $S_1^2>1$.

    By the same method and recalling $S_1^3=4\alpha_1+6\alpha_2+4\alpha_3+\alpha_4$:
    $$E(\Sigma'\Sigma'P\subset\underline{X})\sim n^{8-28\alpha_1-56\alpha_2-70\alpha_3-46\alpha_4}=n^{8-7S_1^3 -14\alpha_2-42\alpha_3-39\alpha_4}.$$
    Assuming $S_1^3$ is sufficiently close to 1, and all $\alpha_i,i\ge 2$ are sufficiently close to 0 (for instance when $\alpha_1=\frac{1}{4},\alpha_{2,3,4}<0.001$), $E(\Sigma'\Sigma'P\subset\underline{X})\to \infty$, as predicted in \Cref{lem:susGood}.    
\begin{figure}
    \centering
\begin{tikzpicture}[scale=2, line join=round]
    % Hexagon vertices
    \foreach \i in {0,...,5} {
        \coordinate (H\i) at ({cos(30+60*\i)}, {sin(30+60*\i)});
    }
    \foreach \i in {0,...,5} {
        \node  at ({1.1*cos(30+60*\i)}, {1.1*sin(30+60*\i)}) {\fpeval{round(3*(\i/3-trunc(\i/3)))+1}};
    }

    \draw[thick] (H0) -- (H1) -- (H2) -- (H3) -- (H4) -- (H5) -- cycle;

    % Inner triangle, half-scale toward the center
    \foreach \i/\j in {0/T0, 2/T1, 4/T2} {
        \coordinate (\j) at ({0.65*cos(30+60*\i)}, {0.65*sin(30+60*\i)});
    }
    \foreach \i in {0,1,2} {
        \node at ({0.5*cos(30+120*\i)}, {0.5*sin(30+120*\i)}) {\fpeval{\i+4}};
    }
    \draw[thick] (T0) -- (T1) -- (T2) -- cycle;
    \foreach \i in {0,1,2}% [evaluate=\i as \j using 2*\i]
    {
        \draw[thick] (T\i)--(H\fpeval{2*\i});
        \draw[thick] (H\fpeval{2*\i+1})--(T\i);
    }
    \draw[thick] (T0)--(H5);
    \draw[thick] (T1)--(H1);
    \draw[thick] (T2)--(H3);
\end{tikzpicture}
    \caption{The minimal triangulation of $\R P^2$}
    \label{fig:RP2}
   
\end{figure}
    
\end{example}

\section{Concluding Remarks and Open Problems}\label{discus}

\begin{itemize}
    \item In \Cref{thmSteenrod} we only address the question of whether the Steenrod operations act trivially or not. 
    We might ask a more quantitative question about the behavior of their rank. Certainly it tends to $\infty$ (in the range found in \Cref{thmSteenrod}), and certainly their "relative rank" ($\frac{rank(\sigma)}{\dim(H^d (\underline{X}))}$), for any $d$ such that $S^d _1(\{\alpha_i\}_i)\geq 1$, tends to 0.
    This is since, from similar arguments to the proof of \Cref{thmSteenrod}, the asymptotically largest contributor to $d$-cohomology is always the simplest, which is empty simplexes with a maximal face.\par
    For the question of the asymptotic rate of growth of the rank, we gave an answer depending on minimal triangulations of spaces with $\sigma$ acting non-trivially, for which the author is not aware of any mention in the literature. The answer to the rate of convergence of $\frac{rank(\sigma)}{\dim(H^d (\underline{X}))}$ to 0 should then follow from the prevalence of this minimal triangulation contrasted with that of empty simplexes with a maximal face.
    \item The results in this paper do not cover all possible parameters. For the cup product, the range of $\alpha$'s discussed in \Cref{thmcup} is dense in "$\alpha$-space", but does not include the case where $S^k _1 (\{\alpha_i\})=1$ for some $k$, nor does it include the case where some $\alpha$'s are 0 (or approach 0 slower than $n^{-\alpha}$, for example $\frac{1}{\log(n)}$).\par
    As for \Cref{steenChap}, we only proved that the algebra acts non-trivially for some range of $\alpha$'s, but have not commented on what happens in the complement. Furthermore, we did not say the exact \di s for which the image of $\sigma$ is non-zero (finding the minimal such \di \ is an interesting question), but only proved an upper bound depending on (again) the minimal triangulation of a complex with non-zero $\sigma$.
    \item Similar questions about other models of random simplicial complexes are left for future publications. Furthermore, we do not address the appearance of non-zero secondary operations, tertiary operations and so on  (such as the Massey products of various lengths), although those should depend on strongly connected components in a similar way.
\end{itemize}

\printbibliography

@article{Fowler19,
author = {Fowler, Christopher F.},
title = {Homology of Multi-Parameter Random Simplicial Complexes},
year = {2019},
issue_date = {July      2019},
publisher = {Springer-Verlag},
address = {Berlin, Heidelberg},
volume = {62},
number = {1},
issn = {0179-5376},
url = {https://doi.org/10.1007/s00454-018-00056-9},
doi = {10.1007/s00454-018-00056-9},
journal = {Discrete Comput. Geom.},
month = {7},
pages = {87–127},
numpages = {41}
}

@book{AT,
  author = {Hatcher, Allen},
  isbn = {0-521-79540-0},
  publisher = {Cambridge University Press},
  title = {Algebraic Topology},
  year = {2002}
}

@article{Kaufmann_2021,
	doi = {10.1515/forum-2020-0296},
  
	url = {https://doi.org/10.1515%2Fforum-2020-0296},
  
	year = 2021,
	month = {10},
  
	publisher = {Walter de Gruyter {GmbH}},
  
	volume = {33},
  
	number = {6},
  
	pages = {1507--1526},
  
	author = {Ralph M. Kaufmann and Anibal M. Medina-Mardones},
  
	title = {Cochain level May{\textendash}Steenrod operations},
  
	journal = {Forum Mathematicum}
}

@book{HandbookDnCG3,
author = {Goodman, J.E. and O’Rourke, J. and Tóth, C.D.},
year = {2017},
month = {01},
pages = {1-1928},
title = {Handbook of discrete and computational geometry, third edition},
journal = {Handbook of Discrete and Computational Geometry, Third Edition},
doi = {10.1201/9781315119601}
}

@incollection{Bobrowski_2022,
	doi = {10.1007/978-3-030-91374-8_2},
  
	url = {https://doi.org/10.1007%2F978-3-030-91374-8_2},
  
	year = 2022,
	publisher = {Springer International Publishing},
  
	pages = {59--96},
  
	author = {Omer Bobrowski and Dmitri Krioukov},
  
	title = {Random Simplicial Complexes: Models and Phenomena},
  
	booktitle = {Understanding Complex Systems}
}

@inbook{CostaFarber,
title={Random Simplicial Complexes},
bookTitle={Configuration Spaces: Geometry, Topology and Representation Theory},
year={2016},
author={Michael Farber and Armindo Costa},
publisher={Springer International Publishing},
address={Cham},
pages={129--153},
isbn={978-3-319-31580-5},
doi={10.1007/978-3-319-31580-5_6},
url={https://doi.org/10.1007/978-3-319-31580-5_6}
}

@article{BERGER07,
title = {Iterated wreath product of the simplex category and iterated loop spaces},
journal = {Advances in Mathematics},
volume = {213},
number = {1},
pages = {230-270},
year = {2007},
issn = {0001-8708},
doi = {https://doi.org/10.1016/j.aim.2006.12.006},
url = {https://www.sciencedirect.com/science/article/pii/S0001870806003938},
author = {Clemens Berger},
}

@article{LinialMeshulam,
author = {Linial, Nathan and Meshulam, Roy},
title = {Homological Connectivity Of Random 2-Complexes},
year = {2006},
issue_date = {August 2006},
publisher = {Springer-Verlag},
address = {Berlin, Heidelberg},
volume = {26},
number = {4},
issn = {0209-9683},
url = {https://doi.org/10.1007/s00493-006-0027-9},
doi = {10.1007/s00493-006-0027-9},

journal = {Combinatorica},
month = {aug},
pages = {475–487},
numpages = {13}
}

@article{MeshulamWallach,
author = {Meshulam, R. and Wallach, N.},
title = {Homological connectivity of random k-dimensional complexes},
journal = {Random Structures \& Algorithms},
volume = {34},
number = {3},
pages = {408-417},
doi = {https://doi.org/10.1002/rsa.20238},
url = {https://onlinelibrary.wiley.com/doi/abs/10.1002/rsa.20238},
eprint = {https://onlinelibrary.wiley.com/doi/pdf/10.1002/rsa.20238},
year = {2009}
}

@article{KAHLE09,
title = {Topology of random clique complexes},
journal = {Discrete Mathematics},
volume = {309},
number = {6},
pages = {1658-1671},
year = {2009},
issn = {0012-365X},
doi = {https://doi.org/10.1016/j.disc.2008.02.037},
url = {https://www.sciencedirect.com/science/article/pii/S0012365X08001477},
author = {Matthew Kahle},
}

@book{Kozlov08,
  author       = {Dmitry N. Kozlov},
  title        = {Combinatorial Algebraic Topology},
  series       = {Algorithms and computation in mathematics},
  volume       = {21},
  publisher    = {Springer},
  year         = {2008},
  url          = {https://doi.org/10.1007/978-3-540-71962-5},
  doi          = {10.1007/978-3-540-71962-5},
  isbn         = {978-3-540-73051-4},
  bibsource    = {dblp computer science bibliography, https://dblp.org}
}

@article{FarberMeadNowik22,
author = {Farber, Michael and Mead, Lewis and Nowik, Tahl},
title = {Random simplicial complexes, duality and the critical dimension},
journal = {Journal of Topology and Analysis},
volume = {14},
number = {01},
pages = {1-32},
year = {2022},
doi = {10.1142/S1793525320500387},

URL = { 
    
        https://doi.org/10.1142/S1793525320500387
    
    

},
eprint = { 
    
        https://doi.org/10.1142/S1793525320500387
    
    

}

}

\end{document}